\documentclass[12pt]{amsart}
\usepackage{amsmath,amsthm,amsfonts,amssymb,latexsym,enumerate,url,hyperref}

\begin{document}

\theoremstyle{plain}

\newtheorem{thm}{Theorem}[section]
\newtheorem{lem}[thm]{Lemma}
\newtheorem{conj}[thm]{Conjecture}
\newtheorem{pro}[thm]{Proposition}
\newtheorem{cor}[thm]{Corollary}
\newtheorem{que}[thm]{Question}
\newtheorem{rem}[thm]{Remark}
\newtheorem{defi}[thm]{Definition}
\newtheorem{hyp}[thm]{Hypothesis}

\newtheorem*{thmA}{THEOREM A}
\newtheorem*{corB}{COROLLARY B}

\newtheorem*{thmC}{THEOREM C}
\newtheorem*{conjA}{CONJECTURE A}
\newtheorem*{conjB}{CONJECTURE B}
\newtheorem*{conjC}{CONJECTURE C}

\newtheorem*{thmAcl}{Main Theorem$^{*}$}
\newtheorem*{thmBcl}{Theorem B$^{*}$}

\numberwithin{equation}{section}

\newcommand{\Maxn}{\operatorname{Max_{\textbf{N}}}}
\newcommand{\frat}[1]{\Phi({#1})}
\newcommand{\fratt}[2]{\Phi^{#1}({#2})}
\newcommand{\Syl}[2]{{\rm Syl}_{#1}({#2})}
\newcommand{\dl}{\operatorname{\mathfrak{d}}}
\newcommand{\Con}{\operatorname{Con}}
\newcommand{\cl}{\operatorname{cl}}
\newcommand{\Stab}{\operatorname{Stab}}
\newcommand{\Aut}{\operatorname{Aut}}
\newcommand{\Ker}{\operatorname{Ker}}
\newcommand{\card}[1]{|#1|}
\newcommand{\Irr}[1]{{\rm Irr}({#1})}
\newcommand{\IBr}[1]{{\rm IBr}({#1})}
\newcommand{\SL}{\operatorname{SL}}
\newcommand{\FF}{\mathbb{F}}
\newcommand{\NN}{\mathbb{N}}
\newcommand{\C}[2]{\textbf{C}_{#1}({#2})}
\newcommand{\N}[2]{\textbf{N}_{#1}({#2})}
\newcommand{\F}{\mathbf{F}}
\newcommand{\aplii}[4]{\begin{array}{ccc}
	#1 &\longrightarrow & #2\\
	#3 &\longmapsto & #4
	\end{array}}

\renewcommand{\labelenumi}{\upshape (\roman{enumi})}

\newcommand{\GL}{\operatorname{GL}}
\newcommand{\Sp}{\operatorname{Sp}}
\newcommand{\PGL}{\operatorname{PGL}}
\newcommand{\PSL}{\operatorname{PSL}}
\newcommand{\SU}{\operatorname{SU}}
\newcommand{\PSU}{\operatorname{PSU}}
\newcommand{\PSp}{\operatorname{PSp}}

\newcommand{\Hall}[2]{{\rm Hall}_{#1}({#2})}

\providecommand{\V}{\mathrm{V}}
\providecommand{\E}{\mathrm{E}}
\providecommand{\ir}{\mathrm{Irr_{rv}}}
\providecommand{\Irrr}{\mathrm{Irr_{rv}}}
\providecommand{\re}{\mathrm{Re}}

\def\irrp#1{{\rm Irr}_{p'}(#1)}
\def\Irrs#1{{\rm Irr}_{\sigma'}(#1)}
\def\ibrp#1{{\rm IBr}_{p'}(#1)}
\def\lin#1{{\rm Lin}(#1)}

\newcommand{\Z}[1]{\textbf{Z}({#1})}
\def\Q{{\mathbb Q}}
\def\irr#1{{\rm Irr}(#1)}
\def\ibr#1{{\rm IBr}(#1)}
\def\irra#1{{\rm Irr}_{\rm A}(#1)}
\def\ibra#1{{\rm IBr}_{\rm A}(#1)}
\def \c#1{{\cal #1}}
\def\cent#1#2{{\bf C}_{#1}(#2)}
\def\nor{\trianglelefteq\,}
\def\oh#1#2{{\bf O}_{#1}(#2)}
\def\Oh#1#2{{\bf O}^{#1}(#2)}
\def\zent#1{{\bf Z}(#1)}
\def\det#1{{\rm det}(#1)}
\def\ker#1{{\rm ker}(#1)}
\def\norm#1#2{{\bf N}_{#1}(#2)}
\def\alt#1{{\rm Alt}(#1)}
\def\iitem#1{\goodbreak\par\noindent{\bf #1}}
   \def \mod#1{\, {\rm mod} \, #1 \, }
\def\sbs{\subseteq}

\def\gc{{\bf GC}}
\def\ngc{{non-{\bf GC}}}
\def\ngcs{{non-{\bf GC}$^*$}}
\newcommand{\notd}{{\!\not{|}}}
\newcommand{\Out}{{\mathrm {Out}}}
\newcommand{\Mult}{{\mathrm {Mult}}}
\newcommand{\Inn}{{\mathrm {Inn}}}
\newcommand{\IBR}{{\mathrm {IBr}}}
\newcommand{\IBRL}{{\mathrm {IBr}}_{\ell}}
\newcommand{\IBRP}{{\mathrm {IBr}}_{p}}
\newcommand{\ord}{{\mathrm {ord}}}
\def\id{\mathop{\mathrm{ id}}\nolimits}
\renewcommand{\Im}{{\mathrm {Im}}}
\newcommand{\Ind}{{\mathrm {Ind}}}
\newcommand{\diag}{{\mathrm {diag}}}
\newcommand{\soc}{{\mathrm {soc}}}
\newcommand{\End}{{\mathrm {End}}}
\newcommand{\sol}{{\mathrm {sol}}}
\newcommand{\Hom}{{\mathrm {Hom}}}
\newcommand{\Mor}{{\mathrm {Mor}}}
\newcommand{\St}{{\sf {St}}}
\def\rank{\mathop{\mathrm{ rank}}\nolimits}
\newcommand{\Tr}{{\mathrm {Tr}}}
\newcommand{\tr}{{\mathrm {tr}}}
\newcommand{\Gal}{{\it Gal}}
\newcommand{\Spec}{{\mathrm {Spec}}}
\newcommand{\ad}{{\mathrm {ad}}}
\newcommand{\Sym}{{\mathrm {Sym}}}
\newcommand{\Char}{{\mathrm {char}}}
\newcommand{\pr}{{\mathrm {pr}}}
\newcommand{\rad}{{\mathrm {rad}}}
\newcommand{\abel}{{\mathrm {abel}}}
\newcommand{\codim}{{\mathrm {codim}}}
\newcommand{\ind}{{\mathrm {ind}}}
\newcommand{\Res}{{\mathrm {Res}}}
\newcommand{\Ann}{{\mathrm {Ann}}}
\newcommand{\Ext}{{\mathrm {Ext}}}
\newcommand{\Alt}{{\mathrm {Alt}}}
\newcommand{\AAA}{{\sf A}}
\newcommand{\SSS}{{\sf S}}
\newcommand{\CC}{{\mathbb C}}
\newcommand{\CB}{{\mathbf C}}
\newcommand{\RR}{{\mathbb R}}
\newcommand{\QQ}{{\mathbb Q}}
\newcommand{\ZZ}{{\mathbb Z}}
\newcommand{\KK}{{\mathbb K}}
\newcommand{\NB}{{\mathbf N}}
\newcommand{\ZB}{{\mathbf Z}}
\newcommand{\OB}{{\mathbf O}}
\newcommand{\EE}{{\mathbb E}}
\newcommand{\PP}{{\mathbb P}}
\newcommand{\GC}{{\mathcal G}}
\newcommand{\HC}{{\mathcal H}}
\newcommand{\AC}{{\mathcal A}}
\newcommand{\BC}{{\mathcal B}}
\newcommand{\GA}{{\mathfrak G}}
\newcommand{\SC}{{\mathcal S}}
\newcommand{\TC}{{\mathcal T}}
\newcommand{\DC}{{\mathcal D}}
\newcommand{\LC}{{\mathcal L}}
\newcommand{\RC}{{\mathcal R}}
\newcommand{\CL}{{\mathcal C}}
\newcommand{\EC}{{\mathcal E}}
\newcommand{\GCD}{\GC^{*}}
\newcommand{\TCD}{\TC^{*}}
\newcommand{\FD}{F^{*}}
\newcommand{\GD}{G^{*}}
\newcommand{\HD}{H^{*}}
\newcommand{\hG}{\hat{G}}
\newcommand{\hP}{\hat{P}}
\newcommand{\hQ}{\hat{Q}}
\newcommand{\hR}{\hat{R}}
\newcommand{\GCF}{\GC^{F}}
\newcommand{\TCF}{\TC^{F}}
\newcommand{\PCF}{\PC^{F}}
\newcommand{\GCDF}{(\GC^{*})^{F^{*}}}
\newcommand{\RGTT}{R^{\GC}_{\TC}(\theta)}
\newcommand{\RGTA}{R^{\GC}_{\TC}(1)}
\newcommand{\Om}{\Omega}
\newcommand{\eps}{\epsilon}
\newcommand{\varep}{\varepsilon}
\newcommand{\al}{\alpha}
\newcommand{\chis}{\chi_{s}}
\newcommand{\sigmad}{\sigma^{*}}
\newcommand{\PA}{\boldsymbol{\alpha}}
\newcommand{\gam}{\gamma}
\newcommand{\lam}{\lambda}
\newcommand{\la}{\langle}
\newcommand{\ra}{\rangle}
\newcommand{\hs}{\hat{s}}
\newcommand{\htt}{\hat{t}}
\newcommand{\sgn}{\mathsf{sgn}}
\newcommand{\SR}{^*R}
\newcommand{\tn}{\hspace{0.5mm}^{t}\hspace*{-0.2mm}}
\newcommand{\ta}{\hspace{0.5mm}^{2}\hspace*{-0.2mm}}
\newcommand{\tb}{\hspace{0.5mm}^{3}\hspace*{-0.2mm}}
\def\skipa{\vspace{-1.5mm} & \vspace{-1.5mm} & \vspace{-1.5mm}\\}
\newcommand{\tw}[1]{{}^#1\!}
\renewcommand{\mod}{\bmod \,}

\newcommand{\I}[2]{{\rm I}_{#1}({#2})}
\newcommand{\Lin}[1]{{\rm Lin}({#1})}
\newcommand{\LinI}[2]{{\rm LinI}_{#1}({#2})}
\newcommand{\LinB}[2]{{\rm LinB}_{#1}({#2})}
\newcommand{\B}[2]{{\rm B}_{#1}({#2})}
\newcommand{\ov}[1]{\overline{#1}}
\newcommand{\OO}[2]{\mathbf{O}_{#1}({#2})}
\renewcommand{\O}[2]{\mathbf{O}^{#1}({#2})}
\def\X{\mathcal{X}}

\newcommand{\Irre}[1]{{\rm Irr}(#1)}
\newcommand{\Irra}[2]{{\rm Irr}_{#1}(#2)}
\def\st{\text{ }|\text{ }}
\newcommand{\inv}[1]{{#1}^{-1}}

\marginparsep-0.5cm

\renewcommand{\thefootnote}{\fnsymbol{footnote}}
\footnotesep6.5pt

\title[Restriction of $p'$-degree characters and relative normal complements]{Restriction of $p'$-degree characters, fields of values, and relative normal complements}
\author{David Cabrera-Berenguer}
\address{Departament de Matem\`atiques, Universitat de Val\`encia, 46100 Burjassot,
Val\`encia, Spain}
\email{david.cabrera@uv.es}

\thanks{This research is supported by Grant PID2022-137612NB-I00
 funded by MCIN/AEI/ 10.13039/501100011033 and ERDF ``A way of making Europe”. The author would like to thank Gabriel A. L. Souza for helpful conversations on the subject.
}



\begin{abstract}
Let $G$ be a finite $p$-solvable group for a prime $p$, and let $P$ be a Sylow $p$-subgroup of $G$. Centella and Navarro proved that there is a canonical bijection ${\rm Irr}_{p'}(\Phi_{1_G})\to{\rm Irr}_{p'}(\Phi_{1_{\textbf N_G(P)}})$ preserving fields of values, where $\Phi_{1_G}$ denotes the principal projective indecomposable character of $G$. In the spirit of a result of Berkovich on the McKay bijection, we characterize when this projective correspondence can be realized by restriction, proving that it holds if and only if $\mathbf{N}_G(P)$ admits a certain relative normal complement in $G$. More generally, we establish a refined criterion incorporating fields of values, and we extend this characterization to $\pi$-separable groups with nilpotent Hall $\pi$-subgroups.
\end{abstract}

\maketitle

\section{Introduction}
A fundamental theme in the representation theory of finite groups is the local-global principle, which aims to predict global character properties from the local structure of a group, mainly from the structrue of a Sylow normalizer. Let $G$ be a finite group and let $P$ be a Sylow $p$-subgroup of $G$ for a prime $p$. In \cite{I86}, Isaacs proved  that if each character of $\Irr{\N GP}$ extends to $G$, then $\N G P$ has a normal complement in $G$. Namely, there exists $N\unlhd G$ such that $G=\N GPN$ and $\N N P=1$. Later, in \cite{B98}, Berkovich refined this criterion by showing that it suffices that each character of ${\N G P}$ with $p'$-degree extends to $G$. Under the lens of the McKay conjecture, this is equivalent to the restriction map
\[\Irra{p'}G\to\Irra{p'}{\N G P}\]
being a bijection, where $\Irra{p'}G$ is the set
of the irreducible characters of $G$ of degree not divisible by $p$.

Let $\Phi_{1_G}$ denote the projective indecomposable character of $G$ associated with the trivial $p$-Brauer character. If $G$ is $p$-solvable, then Centella and Navarro proved in \cite{CN08} that there is a canonical bijection
\[\Irra{p'}{\Phi_{1_G}}\to\Irra{p'}{\Phi_{1_{\N G P}}},\]
preserving fields of values, where $\Irra{p'}{\Phi_{1_G}}$ is the set of irreducible constituents of $\Phi_{1_G}$ with $p'$-degree. This bijection sends each $\chi\in\Irra{p'}{\Phi_{1_G}}$ to the unique irreducible constituent of $p'$-degree of $\chi_{\N G P}$.
However, as they noted, the equality $\card{\Irra{p'}{\Phi_{1_G}}}=\card{\Irra{p'}{\Phi_{1_{\N G P}}}}$ does not necessarily hold outside the $p$-solvable groups (take for instance $G=A_5$ and $p=2$).

Since the McKay bijection being realized by a restriction map yields the existence of a normal complement of $\N G P$ in $G$, one may wonder whether the same property for the previous bijection implies that $\N G P$ has a relative normal complement in $G$. We show that this is indeed the case. In fact, we prove a broader result incorporating fields of values, which extends naturally to $\pi$-separable groups with a nilpotent Hall $\pi$-subgroup.

To work with $\Irra{p'}{\Phi_{1_G}}$, we will use Isaacs' $\pi$-theory of $\pi$-separable groups. Suppose that $G$ is a $\pi$-separable group, and denote by $G^0$ the set of $\pi$-elements of $G$. Then a map $\varphi:G^0\to\mathbb C$ is a $\pi$-partial character if $\varphi$ is the restriction of some character $\chi$ of $G$ to $G^0$. If $\varphi$ cannot be expressed as a sum of two $\pi$-partial characters, then we say that $\varphi$ is irreducible. The set of irreducible $\pi$-partial characters of $G$ is denoted by $\I\pi G$. We will work with the canonical set of lifts $\B\pi G\subseteq\Irr G$ of $\I\pi G$ to $G$ (for more details, see \cite{I84} or \cite{I18}). This is because, as we shall see later, if $G$ is $p$-solvable then $\Irra{p'}{\B p G}=\Irra{p'}{\Phi_{1_G}}$.

If $n$ is a positive integer, we denote $\Q_n=\Q(e^{2\pi i/n})$ the $n$-th cyclotomic field. We write
\[\Irra{\mathbb Q_n}G=\{\chi\in\Irr G:\mathbb Q(\chi)\subseteq\mathbb Q_n\},\]
\[\Irra{p',\mathbb Q_n}{\B p G}=\B p G\cap\Irra{p'}G\cap\Irra{\mathbb Q_n}G.\]

If $P$ is a finite $p$-group and $n$ is a positive integer, then we define $\fratt n P$ as the unique smallest normal subgroup of $P$ such that $P/\fratt n P$ is isomorphic to a direct product of cyclic $p$-subgroups of order at most $p^n$. Notice that $\fratt 1 P=\frat P$, and by the fundamental theorem of finite abelian groups there is a positive integer $k$ with $P'=\fratt k P$ (for instance, we may take $k=\log_p\card P$). Thus,
\[P'=\fratt k P\subseteq\fratt{k-1}P\subseteq\ldots\subseteq\fratt 2 P\subseteq\frat P.\]
By the definition it easily follows that $\Irr{P/\fratt nP}$ is the set of linear characters of $P$ of order at most $p^n$. While our primary results are established for a single prime $p$, we present in the final section of the paper how these theorems generalize to the context of $\pi$-separable groups with a nilpotent Hall $\pi$-subgroup.

We prove the following.
\begin{thmA}
	Let $G$ be $p$-solvable, let $P\in\Syl p G$ and let $n$ be a positive integer. Write $K_n=\bigcap_{\chi\in\Irra{p',\mathbb Q_{p^n}}{\B p G}}\ker\chi$. Then the following are equivalent.
    \begin{enumerate}[a)]
        \item The restriction map $\Irra{p',\mathbb Q_p}{\B p G}\to \Irra{p',\mathbb Q_p}{\B p{\N G P}}$ is a bijection.
        \item It holds that $G=\N G PK_n$.
    \end{enumerate}
    Moreover, if the previous holds, then $G$ has $p$-length $1$ and \[\N G P\cap K_n=\fratt nP\OO{p'}{\N G P}.\] In general, the relative complement $K_n$ of $\N G P$ is unique.
\end{thmA}
Indeed, condition a) of Theorem A is equivalent to that every $\chi\in\Irra{p',\mathbb Q_p}{\B p G}$ restricts irreducibly to $\N G P$, as we will see.
\section{Preliminaries}
Our notation for characters follows \cite{I06} and \cite{N18}. 
\begin{lem}\label{quoPre}
    Let $G$ be $\pi$-separable and let $N\unlhd G$. Let $\chi\in\Irr G$ be such that $N\subseteq\ker\chi$, and consider $\overline\chi$ its corresponding character in $G/N$. Then $\chi\in\B\pi G$ if and only if $\overline\chi\in\B\pi{G/N}$.
\end{lem}
\begin{proof}
    See Problem (4.7) of \cite{I18}.
\end{proof}
\begin{lem}\label{interKernel}
Let $G$ be $\pi$-separable. Then $\bigcap_{\chi\in\B\pi G}\ker\chi=\OO{\pi'} G$.
\end{lem}
\begin{proof}
    See Corollary (4.22) of \cite{I18}.
\end{proof}
\begin{thm}\label{fong}
    Let $G$ be $\pi$-separable, let $H$ be a Hall $\pi$-subgroup of $G$, and let $\chi\in\B\pi G$. Let $\alpha$ be a minimal degree constituent of $\chi_H$. Then the following hold.
    \begin{enumerate}[a)]
        \item $\chi(1)_\pi=\alpha(1)$ and $[\chi_H,\alpha]=1$.
        \item $[\psi_H,\alpha]=0$ for each $\psi\in\B\pi G$ different from $\chi$.
    \end{enumerate}
\end{thm}
\begin{proof}
    See Theorem (3.4) of \cite{I18}, then use the bijection of Theorem (5.1) of \cite{I18}.
\end{proof}
In the notation of the previous theorem, if $\alpha\in\Irr H$ is a minimal degree constituent of $\chi_H$, where $\chi\in\B\pi G$, then we say that $\alpha$ is a Fong character associated to $\chi$.
\begin{lem}\label{fongfield}
    Let $G$ be $\pi$-separable and let $H$ be a Hall $\pi$-subgroup of $G$. If $\lambda\in\Irr H$ is a Fong character for some $\chi\in\B\pi G$, then $\Q(\chi)\subseteq\Q(\lambda)$.
\end{lem}
\begin{proof}
    It is a direct consequence theorems (3.19) and (5.1) of \cite{I18}.
\end{proof}
\begin{thm}\label{isFong}
    Let $G$ be $\pi$-separable and let $H$ be Hall $\pi$-subgroup of $G$. If $\alpha\in\Irr H$ is quasiprimitive, then $\alpha$ is a Fong character.
\end{thm}
\begin{proof}
    See Theorem (5.6) of \cite{I18}.
\end{proof}
\begin{thm}\label{centnav}
    Let $G$ be $\pi$-separable and let $K$ be a $\pi$-complement of $G$. Let $\chi\in\Irr G$ be such that $\chi(1)$ is a $\pi'$-number. Then $\chi\in\B\pi G$ if and only if $[\chi_K,1_K]>0$.
\end{thm}
\begin{proof}
    See Theorem (2.2) of \cite{CN08}.
\end{proof}
If $K$ is a $p$-complement of a group $G$, then it holds that $\Phi_{1_G}=(1_K)^G$ (see for instance Theorem (2.27) of \cite{N98}). Hence, if $G$ is $p$-solvable, from the previous theorem and Frobenius reciprocity it follows that
\[\Irra{p'}{\B p G}=\Irra{p'}{\Phi_{1_G}}.\]
\begin{thm}(Tate)
    Let $H$ be a nilpotent Hall $\pi$-subgroup of $G$ and let $N\unlhd G$. If $H\cap N\subseteq \frat H$, then $N$ has a normal $\pi$-complement.
\end{thm}
\begin{proof}
    Write $H=P_1\times\ldots\times P_k$ for its distinct Sylow subgroups. Fix $1\leq i\leq k$ and write $P=P_i$ for the Sylow $p$-subgroup. Since $H$ is nilpotent, we have $\frat P\in\Syl p{\frat H}$. As $P\cap N\subseteq\frat H$, then necessarily $P\cap N\subseteq\frat P$. By the main result of \cite{T64} it follows that $N$ is $p$-nilpotent. Taking now the intersection of all the normal $p$-complements of $N$, where $p$ runs over $\pi$, we obtain that $N$ has a normal $\pi$-complement.
\end{proof}
\begin{lem}\label{thompsonsep}
    Let $G$ be $\pi$-separable and let $H$ be a Hall $\pi$-subgroup of $G$. Then $\OO{\pi'}{\N G H}\subseteq\OO{\pi'}G$.
\end{lem}
\begin{proof}
We proceed by induction on $\card G$. Write $N=\OO{\pi'}G$ and suppose that $N>1$. As $\OO{\pi'}{G/N}=1$, by induction it follows that $\OO{\pi'}{\N{G/N}{HN/N}}=1$.
    Notice that $\OO{\pi'}{\N G H}\subseteq\C G H$.  Then by Hall-Higman 1.2.3 lemma it follows that
    \[\OO{\pi'}{\N GH}N/N\subseteq\C{G/N}{\OO\pi{\N{G/N}{HN/N}}}\subseteq\OO\pi{\N{G/N}{HN/N}}.\]
    Hence $\OO{\pi'}{\N G H}N/N$ is a $\pi$-group and therefore it has to be trivial, yielding that $\OO{\pi'}{\N GH}\subseteq N$. If $N=1$, then using again Hall-Higman 1.2.3 lemma it follows that
    \[\OO{\pi'}{\N G H}\subseteq\C G H\subseteq\C G{\OO\pi G}\subseteq\OO\pi G,\]
    yielding $\OO{\pi'}{\N G H}=1$.
\end{proof}
\begin{lem}\label{abhallnormal}
    Let $G$ be $\pi$-separable and let $H$ be a Hall $\pi$-subgroup of $G$. If $H$ is abelian, then $H\OO{\pi'}G\unlhd G$.
\end{lem}
\begin{proof}
Write $N=\OO{\pi'}{G}$.
    Since $\OO{\pi'}{G/N}=1$, by Hall-Higman 1.2.3 Lemma we have that
    \[\C{G/N}{\OO\pi{G/N}}\subseteq\OO\pi{G/N}\subseteq HN/N.\]
    However, since $H$ is abelian, so is $HN/N$, and hence $HN/N\subseteq\C{G/N}{\OO\pi{G/N}}$. Then $HN/N=\OO\pi{G/N}$ and in particular $HN\unlhd G$.
\end{proof}
\begin{lem}\label{tatefrat}
    Let $P\in\Syl p G$ and suppose that each linear character of order $p$ of $P$ extends to $G$. Then $G$ has a normal $p$-complement.
\end{lem}
\begin{proof}
    Let $\textbf{E}^p(G)$ be the unique smallest normal subgroup of $G$ such that $G/\textbf{E}^p(G)$ is an elementary abelian $p$-group. By Problem (6.20) of \cite{I06} it suffices to check that $\Phi(P)=\textbf{E}^p(G)\cap P$. Let $X$ be the set of linear characters of $G$ of order dividing $p$. We check that
    \[\textbf{E}^p(G)=\bigcap_{\chi\in X}\ker\chi\]
Since $\Irr{G/\textbf{E}^p(G)}\subseteq X$, one inclusion is clear, and the other one follows since $G/\bigcap_{\chi\in X}\ker\chi$ embeds into the $p$-elementary abelian group $\prod_{\chi\in X}G/\ker\chi$. Now we see that the restriction map $X\to\Irr{P/\frat P}$ is a bijection. Obviously it is well-defined. If $\lambda\in\Irr{P/\frat P}$, then by hypothesis there is $\hat\lambda\in\Irr G$ extending $G$, which we may assume to be of $p$-power order. Therefore $G=P\ker{\hat\lambda}$, and hence $\hat\lambda$ is the corresponding extension of $\lambda$ under the isomorphism $P\ker{\hat\lambda}/\ker{\hat\lambda}\cong P/\ker\lambda$, yielding $o(\hat\lambda)=o(\lambda)$. If $\alpha,\beta\in X$ are such that $\alpha_P=\beta_P$, then $P\subseteq\ker{\alpha\inv\beta}$. However, since $o(\alpha\inv\beta)$ is a $p$-number it follows that $o(\alpha\inv\beta)=1$. Thus,
\[\textbf{E}^p(G)\cap P=\bigcap_{\chi\in X}\ker\chi\cap P=\bigcap_{\chi\in X}\ker{\chi_P}=\bigcap_{\lambda\in\Irr{P/\frat P}}\ker\lambda=\frat P.\]
\end{proof}
Notice that from the previous lemma it easily follows that if $H$ is a nilpotent Hall $\pi$-subgroup of $G$ and every linear character of $H$ of square-free order extends to $G$, then $G$ has a normal $\pi$-complement.
\section{Proofs}
\begin{lem}
    Let $G$ be $p$-solvable, let $P\in\Syl p G$ and let $n$ be a positive integer. Then the subgroup
    \[L=\bigcap_{\substack{\chi\in\Irra{p',\mathbb Q_{p^n}}{\B p G}\\\chi(1)>1}}\ker\chi\]
    is $p$-nilpotent with normal $p$-complement $\OO{p'}G$.
\end{lem}
\begin{proof}
Let $\lambda\in\Lin P$ be of order at most $p^n$. By Theorem \eqref{isFong} we have that $\lambda$ is a Fong character of some $\alpha_\lambda\in\Irra{p'}{\B p G}$. Using now Lemma \eqref{fongfield} it follows that $\Q(\alpha_\lambda)\subseteq\Q(\lambda)=\mathbb Q_{o(\lambda)}$, and hence $\alpha_\lambda\in\Irra{p',\mathbb Q_{p^n}}{\B p G}$. If $\Irra{p',\mathbb Q_{p^n}}{\B p G}\subseteq\Lin G$, then $\lambda$ exends to $G$, and applying now Lemma \eqref{tatefrat} we obtain that $L=G$ is $p$-nilpotent. Thus we may assume that there is some nonlinear $\psi\in\Irra{p',\mathbb Q_{p^n}}{\B p G}$.
If $\alpha_\lambda(1)>1$, then $L\subseteq\ker{\alpha_\lambda}$. Suppose that $\alpha_\lambda$ is linear and let $K$ be a $p$-complement of $G$. As $\Irra{p'}{\B p G}=\Irra{p'}{\Phi_{1_G}}$ and $\Phi_{1_G}=(1_K)^G$, we have that $K\subseteq\ker{\alpha_\lambda}$. Therefore, $[(\alpha_\lambda\psi)_K,1_K]=[\psi_K,1_K]>0$, and hence by Theorem \eqref{centnav} it follows that $\alpha_\lambda\psi\in\Irra{p',\mathbb Q_{p^n}}{\B p G}$. As $L\subseteq\ker{\alpha_\lambda\psi}$, it holds that $L\subseteq\ker{\alpha_\lambda}$. Therefore $L\cap P\subseteq\ker\lambda$, yielding $L\cap P\subseteq\fratt n P$. Using now Tate's theorem we have that $L$ has a normal $p$-complement.
    
    By Lemma \eqref{interKernel} we have that $\OO{p'}G\subseteq L$. Since $L$ is $p$-nilpotent then its normal $p$-complement $N$ is contained in $\OO{p'}G$, and hence necessarily $N=\OO{p'}G$.
\end{proof}
From the previous proof, we obtain the following.
\begin{lem}\label{berkobp}
    Let $G$ be $p$-solvable and let $P\in\Syl p G$. Then the subgroup
    \[L=\bigcap_{\chi\in\Irra{p',\mathbb Q_{p^n}}{\B p G}}\ker\chi\]
    is $p$-nilpotent with normal $p$-complement $\OO{p'}G$ and $P\cap L\subseteq\fratt n P$.
\end{lem}
\begin{cor}\label{intbpprime}
    Let $G$ be $p$-solvable, let $P\in\Syl p G$ and let $n$ be a positive integer. Then
    \[\bigcap_{\psi\in\Irra{p',\mathbb Q_{p^n}}{\B p{\N G P}}}\ker\psi=\fratt n P\OO{p'}{\N G P}.\]
\end{cor}
\begin{proof}
    We claim that $\Irra{p',\mathbb Q_{p^n}}{\N G P}\subseteq \Irr{\N G P/\fratt n P}$. Suppose that $\psi\in\Irra{p',\mathbb Q_{p^n}}{\N G P}$. Notice that we may assume that $p^n\leq \card P$, and let $\mathcal G={\rm Gal}(\mathbb Q_{\card P}/\mathbb Q_{p^n})$. Then $\mathcal G$ is a $p$-group acting on the constituents of $\psi_P$. By Clifford's theorem we have that the number of constituents of $\psi_P$ is a $p'$-number, and hence there is some $\mathcal G$-invariant constituent $\lambda$ of $\psi_P$. Since $\mathbb Q_{\card P}/\Q_{p^n}$ is Galois, we have $\mathbb Q_{o(\lambda)}=\Q(\lambda)\subseteq\Q_{p^n}$, and hence $o(\lambda)$ divides $p^n$. Thus, $\fratt n P\subseteq\ker\lambda$. Also, since $\Q(\lambda)=\Q(\lambda^x)$ for each $x\in\N G P$, by Clifford's theorem we have that $\fratt n P\subseteq\ker{\psi_P}\subseteq\ker\psi$, yielding the claim.

    Now, write $L=\bigcap_{\psi\in\Irra{p',\Q_{p^n}}{\B p{\N G P}}}\ker\psi$. Then,
    \[\fratt n P\subseteq\bigcap_{\psi\in\Irra{p',\Q_{p^n}}{\N G P}}\ker\psi\subseteq\bigcap_{\psi\in\Irra{p',\Q_{p^n}}{\B p{\N G P}}}\ker\psi=L.\]
    By Lemma \eqref{berkobp} we have that $\OO{p'}{\N G P}$ is a normal $p$-complement of $L$ and $P\cap L\subseteq\fratt n P$, yielding the result.
\end{proof}
\begin{thm}\label{uniquepconst}
    Let $G$ be $\pi$-separable, let $H$ be a Hall $\pi$-subgroup of $G$ and let $\chi\in\Irra{\pi'}{\B \pi G}$. Then $\chi_{\N GH}$ has a unique $\pi'$-degree constituent $\chi^*$, which lies in $\Irra{\pi'}{\B\pi{\N GH}}$. Also, the map $\Irra{\pi'}{\B\pi G}\to\Irra{\pi'}{\B\pi{\N G H}}$ defined via $\chi\mapsto\chi^*$ is a bijection preserving fields of values.
\end{thm}
\begin{proof}
    See Theorem (2.4) of \cite{CN08} and the subsequent discussion.
\end{proof}
Notice that the previous theorem implies that condition a) of Theorem A is equivalent to that every $\chi\in\Irra{p',\mathbb Q_p}{\B p G}$ restricts irreducibly to $\N G P$.
\begin{thm}\label{largestbpprime}
    Let $G$ be $p$-solvable, let $P\in\Syl p G$ and let $n$ be a positive integer. Then $K=\bigcap_{\chi\in\Irra{p',\mathbb Q_{p^n}}{\B p G}}\ker\chi$ is the unique largest normal subgroup $K$ of $G$ such that $\N K P\subseteq\fratt n P\OO{p'}G$.
\end{thm}
\begin{proof}
By Lemma (2.1) of \cite{IN05} there is a unique largest normal subgroup satisfying the property.

For each $\psi\in\Irra{p',\mathbb Q_{p^n}}{\B p{\N G P}}$ let $\chi_{\psi}\in\Irra{p',\mathbb Q_{p^n}}{\B p G}$ be preimage of $\psi$ under the bijection of Theorem \eqref{uniquepconst}. Then, $\{\chi_\psi\}_{\psi\in\Irra{p',\mathbb Q_{p^n}}{\B p{\N G P}}}=\Irra{p',\mathbb Q_{p^n}}{\B p G}$. Hence, by Corollary \eqref{intbpprime} we deduce that
\[K\cap\N G P\subseteq\bigcap_{\psi\in\Irra{p',\mathbb Q_{p^n}}{\B p{\N G P}}}\ker\psi=\fratt n P\OO{p'}{\N G P}.\]
Since $G$ is $p$-solvable, by Lemma \eqref{thompsonsep} we have $\OO{p'}{\N G P}\subseteq\OO{p'}G$, and hence $K$ satisfies the condition.

Now, let $N\unlhd G$ be such that $\N N P\subseteq\fratt n P\OO{p'}{G}$. Intersecting with $P$ we obtain that $P\cap N\subseteq\fratt n P\subseteq\frat P$, and hence by Tate's theorem we have that $N$ is $p$-nilpotent. Let $Q$ be a normal $p$-complement of $N$, and notice that $Q\subseteq\OO{p'}G\subseteq K$. As $Q$ is a normal $p'$-group, then
\[\fratt n{PQ/Q}\cong\fratt n P\cong\fratt n PQ/Q.\]
Write $\fratt n{PQ/Q}=X/Q$ and notice that $PQ/X$ is isomorphic to a direct product of cyclic grups of order at most $p^n$. Then so is $PX/X\cong P/P\cap X$, and hence $\fratt nP\subseteq X$, yielding $\fratt n{PQ/Q}=\fratt n PQ/Q$. Working in $G/Q$ and using Lemma \eqref{quoPre}, we may assume that $N$ is a $p$-group. Now, let $\chi\in\Irra{p',\mathbb Q_{p^n}}{\B p G}$ and let $\chi^*\in\Irra{p',\mathbb Q_{p^n}}{\B p{\N G P}}$ be the unique $p'$-degree constituent of $\chi_{\N G P}$, which exists by Theorem \eqref{uniquepconst}. By Corollary \eqref{intbpprime} it follows that $\fratt n P\subseteq\ker{\chi^*}$. Since $N\subseteq\fratt n P$, by Clifford's theorem it holds that $N\subseteq\ker\chi$, and hence $N\subseteq K$.
\end{proof}
We can now prove Theorem A.
\begin{thm}\label{mainthm}
    Let $G$ be $p$-solvable, let $P\in\Syl p G$ and let $n$ be a positive integer. Write $K_n=\bigcap_{\chi\in\Irra{p',\mathbb Q_{p^n}}{\B p G}}\ker\chi$. Then the following are equivalent.
    \begin{enumerate}[a)]
        \item The restriction map $\Irra{p',\mathbb Q_p}{\B p G}\to \Irra{p',\mathbb Q_p}{\B p{\N G P}}$ is a bijection.
        \item It holds that $G=\N G PK_n$.
    \end{enumerate}
    Moreover, if the previous holds, then $G$ has $p$-length $1$ and \[\N G P\cap K_n=\fratt nP\OO{p'}{\N G P}.\] In general, the relative complement $K_n$ of $\N G P$ is unique.
\end{thm}
\begin{proof}
    First we prove the result for $n=1$. Write $K=K_1$, suppose that the restriction map $\Irra{p',\mathbb Q_p}{\B p G}\to \Irra{p',\mathbb Q_p}{\B p{\N G P}}$ is a bijection and let $\chi\in\B p{G/K}$. By Corollary \eqref{intbpprime} we have that $\N G P\cap K=\frat P\OO{p'}{\N G P}$. Thus $\frat P\subseteq\ker{\chi_{\N G P}}$ and therefore $\chi_P$ contains just linear constituents of order at most $p$. By Lemma \eqref{quoPre} we have that $\chi\in\B p G$. Since the Fong characters of $\chi$ are linear, using now Theorem \eqref{fong} it holds that $\chi$ has $p'$-degree. Also, they have order at most $p$, and hence by Lemma \eqref{fongfield} it follows that $\Q(\chi)\subseteq\Q_{p}$, yielding $\chi\in\Irra{p',\mathbb Q_p}{\B p G}$. Therefore, by hypothesis we obtain that $\chi_{\N G P}\in\Irr{\N G P}$, and hence $\chi_{\N G PK}\in\Irr{\N G PK}$. Viewing now $\chi$ as a character of $G/K$ we have that
    $(PK/K)'\subseteq\ker\chi\cap\N G PK/K$, and since this holds for an arbitrary $\chi\in\B p{G/K}$, by Lemma \eqref{interKernel} we have that $(PK/K)'\subseteq\OO{p'}{G/K}$. Thus $PK/K$ is abelian, and hence applying Lemma \eqref{abhallnormal} we obtain that $(PK/K)\OO{p'}{G/K}\unlhd G$. However, as $\Irra{p',\mathbb Q_p}{\B p G}=\B p{G/K}$, by applying Lemma \eqref{quoPre} and Lemma \eqref{interKernel} it follows that $K/K=\OO{p'}{G/K}$, and hence $PK\unlhd G$, yielding $G=\N G PK$. By Lemma \eqref{berkobp} it holds that $K$ is $p$-nilpotent with normal $p$-complement $\OO{p'}G$. Therefore,
    \[G=\N G PK=\N GP(P\cap K)\OO{p'}G=\N G P\OO{p'}G,\]
    and hence $G$ has $p$-length $1$.

    Suppose now that $G=\N G PK$ and let $\chi\in\Irra{p',\mathbb Q_p}{\B p G}$. Since $K\subseteq\ker\chi$ and $G=\N G PK$, by elementary character theory we have that $\chi_{\N G P}\in\Irr{\N G P}$. Applying now Theorem \eqref{uniquepconst} it follows that the restriction map $\Irra{p',\mathbb Q_p}{\B p G}\to\Irra{p',\mathbb Q_p}{\B p{\N G P}}$ is a bijection.

    Suppose now that $n$ is an arbitrary positive integer and a) holds. Then $G$ has $p$-length $1$, and by Lemma \eqref{interKernel} we obtain that the restriction map
    \[\Irra{p',\mathbb Q_{p^n}}{\B p G}\to\Irra{p',\mathbb Q_{p^n}}{\B p{\N G P}}\]
    is a bijection. Proceeding as in the first part of the proof we obtain $G=\N G P K_n$ and $\N G P\cap K_n=\fratt n P\OO{p'}{\N G P}$. Also, if $G=\N G PK_n$, then $G=\N G PK$, as $K_n\subseteq K$.

    Now we prove the unicity of the relative normal complements $K_n$. Let $n$ be a positive integer and suppose that there is $N\unlhd G$ such that
     \[G=\N G PN\text{ and }\N G P\cap N=\fratt n P\OO{p'}{\N G P}.\]
Since $G$ is $p$-solvable, by Lemma \eqref{thompsonsep} we have $\N G P\cap N\subseteq\fratt n P\OO{p'}G$, and hence by Theorem \eqref{largestbpprime} it follows that $N\subseteq K_n$, yielding $G=\N G P K_n$. By the same theorem we have that $\N G P\cap K_n\subseteq\fratt n P\OO{p'}G$, and hence intersecting with $\N G P$ we obtain that
    \[\N G P\cap K_n\subseteq\fratt n P\OO{p'}G\cap\N G P=\fratt n P(\OO{p'}G\cap\N G P)=\fratt n P\OO{p'}{\N G P}.\]
    Thus, $\N G P\cap K_n=\fratt n P\OO{p'}{\N G P}$ and hence comparing cardinals we obtain that $N=K_n$.
\end{proof}
Suppose that $G$ is $p$-solvable and let $P\in\Syl p G$. In \cite{CN08}, Centella and Navarro proved that the following inequality holds.
\[\sum_{\chi\in\Irra{p'}{\B p G}}\chi(1)\geq\card{P:P'}.\]
Using Theorem A of the same paper and lemmas \eqref{fong}, \eqref{fongfield} and \eqref{isFong}, it is not difficult to check that
\[\sum_{\chi\in\Irra{p',\mathbb Q_{p^n}}{\B p G}}\chi(1)\geq\card{P:\fratt n P}\]
for each positive integer $n$.
We characterize when the equality holds.
\begin{cor}
     Let $G$ be $p$-solvable, let $P\in\Syl pG$ and let $n$ be a positive integer. Then the following are equivalent.
    \begin{enumerate}[a)]
        \item $\sum_{\chi\in\Irra{p',\mathbb Q_{p^n}}{\B p G}}\chi(1)=\card{P:\fratt n P}.$
        \item There is $N\unlhd G$ such that $G=\N G PN$ and $\N G P\cap N=\fratt n P\OO{p'}{\N G P}$.
    \end{enumerate}
\end{cor}
\begin{proof}
    If the equality holds, then by the previous observation we obtain that \[\sum_{\chi\in\Irra{p',\mathbb Q_{p^n}}{\B p G}}\chi(1)\leq\sum_{\psi\in\Irra{p',\mathbb Q_{p^n}}{\B p{\N G P}}}\psi(1).\]
    Using now Theorem \eqref{uniquepconst} it follows that each $\chi\in\Irra{p',\mathbb Q_{p^n}}{\B p G}$ restricts irreducibly to $\N G P$, and hence by Theorem \eqref{mainthm} we are done in this case.
    
    Suppose now that there is $N\unlhd G$ with 
    \[G=\N G PN \text{ and } \N G P\cap N=\fratt n P\OO{p'}{\N G P}.\]
    By Lemma \eqref{mainthm} we have that $N=K$, where $K=\bigcap_{\chi\in\Irra{p',\mathbb Q_{p^n}}{\B p G}}\ker\chi$. Proceeding as in the proof of Theorem \eqref{mainthm} we obtain that $\B p{G/K}=\Irra{p',\mathbb Q_{p^n}}{\B p G}$. Since $G/K\cong\N G P/{\fratt n P\OO{p'}{\N G P}}$, the restriction map
    \[\B p{G/K}\to\B p{\N G P/{\fratt n P\OO{p'}{\N G P}}}\]
    is a bijection. Applying now Theorem (3.18) of \cite{I18} we obtain
    \[\begin{aligned}
        \sum_{\chi\in\Irra{p',\mathbb Q_{p^n}}{\B p G}}\chi(1)&=\sum_{\psi\in\B p{\N G P/{\fratt n P\OO{p'}{\N G P}}}}\psi(1)\\&=\card{\N G P/{\fratt n P\OO{p'}{\N G P}}}_p\\&=\card{P/\fratt n P}.
    \end{aligned}\]
\end{proof}
So far, we have focused on a single prime $p$. To extend our setting to a set of primes $\pi$, we proceed in the following way. Suppose that $H$ is a nilpotent Hall $\pi$-subgroup of a $\pi$-separable group $G$ and write $H=P_1\times\ldots\times P_k$ for its distinct Sylow subgroups, where $P_i\in\Syl{p_i}H$. Let $n_1,\ldots,n_k$ be positive integers, where $k=\card\pi$, and denote $\bar n=(n_1,\ldots,n_k)$. We define $\fratt{\bar n} H=\prod_{i=1}^k\fratt{n_i}{P_i}$. Then $\fratt{\bar n}H$ is the unique smallest normal subgroup of $H$ whose quotient is isomorphic to a direct product of cyclic $p_i$-subgroups of order at most $p_i^{n_i}$, where $1\leq i\leq k$. We also denote $\pi^{\bar n}=\prod_{i=1}^kp_i^{n_i}$, where this value depends on $H$, and $\pi^{\bar 1}=\prod_{i=1}^kp_i$. To obtain the nilpotent Hall $\pi$-subgroup version of Theorem \eqref{mainthm} we need the following lemma.
\begin{lem}
    Let $H$ be a nilpotent Hall $\pi$-subgroup of a group $G$. Let $\bar n=(n_1,\ldots,n_{\card\pi})$ be a sequence of positive integers and let $\psi\in\Irra{\pi',\mathbb Q_{\pi^{\bar n}}}{\N G H}$. Then there exists a constituent $\lambda\in\Irr H$ of $\psi_H$ with $\fratt{\bar n}H\subseteq\ker\lambda$.
\end{lem}
\begin{proof}
    Using the previous notation, we may assume that $p_i^{n_i}\leq\card{P_i}$ for each $i$.
    Let $m=\card G_{\pi'}$ and let $\mathcal G={\rm Gal}(\mathbb Q_{\card G}/\mathbb Q_{\pi^{\bar n}m})$. By elementary Galois theory we have that $\mathcal G$ is a $\pi$-group. Write $\Omega$ for the set of irreducible constiuents of $\psi_H$. Since $H$ acts trivially on $\Omega$, by Clifford's theorem we have that the $\pi'$-group $\N G H/H$ acts transitively on $\Omega$. Then $\mathcal G$ defines a (trivial) coprime action by automorphisms on $\N GH/H$. Now, the Galois action on $\Omega$ commutes with the conjugation action of $\N GH/H$. Hence, by the odd order theorem and Glauberman's lemma (see Lemma (13.8) of \cite{I06}) we have that there is a $\mathcal G$-invariant constituent $\lambda$ of $\psi_H$. Write $\lambda=\lambda_1\times\ldots\times\lambda_{\card\pi}$, where $\lambda_i\in\Irr{P_i}$. As $\mathbb Q_{\card G}/\mathbb Q_{\pi^{\bar n}m}$ is Galois, we have $\Q_{o(\lambda_1)\ldots o(\lambda_n)}\subseteq\mathbb Q_{\pi^{\bar n}m}$. Let $1\leq i\leq \card\pi$. By elementary Galois theory it follows that $\Q_{o(\lambda_i)}\subseteq\mathbb Q_{p_i^{n_i}}$, and hence $o(\lambda_i)$ divides $p_i^{n_i}$. In particular $\Phi^{n_i}(P_i)\subseteq\ker{\lambda_i}$, and by the definition of $\Phi^{\bar n}(H)$ it holds that
    \[\Phi^{\bar n}(H)\subseteq\ker{\lambda_1}\times\ldots\times\ker{\lambda_{\card\pi}}=\ker{\lambda_1\times\ldots\times\lambda_{\card\pi}}=\ker\lambda.\]
\end{proof}
To keep the notation simple, we have presented the results of this section for a single prime $p$. Nevertheless, all of them remain true if we replace the prime $p$ with a set of primes $\pi$, the Sylow $p$-subgroup $P$ with a nilpotent Hall $\pi$-subgroup $H$, each positive integer $n$ with a sequence of positive integers $\bar n=(n_1,\ldots,n_{\card\pi})$, and $\fratt n P$ with $\fratt{\bar n}H$.
The proofs we gave also work for the nilpotent Hall $\pi$-subgroup case, except for that of Lemma \eqref{intbpprime}. However, using the previous lemma it follows immediately. Hence, the following $\pi$-version of Theorem \eqref{mainthm} also holds.
\begin{thm}
    Let $G$ be $\pi$-separable with a nilpotent Hall $\pi$-subgroup $H$, and let $\bar n=(n_1,\ldots,n_{\card\pi})$ be a sequence of positive integers. Write $K_{\bar n}=\bigcap_{\chi\in\Irra{\pi',\mathbb Q_{\pi^{\bar n}}}{\B\pi G}}\ker\chi$. Then the following are equivalent.
    \begin{enumerate}[a)]
        \item The restriction map $\Irra{\pi',\mathbb Q_{\pi^{\bar 1}}}{\B\pi G}\to \Irra{\pi',\mathbb Q_{\pi^{\bar 1}}}{\B\pi{\N G H}}$ is a bijection.
        \item It holds that $G=\N G HK_{\bar n}$.
    \end{enumerate}
    Moreover, if the previous holds, then $G$ has $\pi$-length $1$ and \[\N G H\cap K_{\bar n}=\fratt{\bar n}H\OO{\pi'}{\N G H}.\] In general, the relative complement $K_{\bar n}$ of $\N G H$ is unique.
\end{thm}


\begin{thebibliography}{ABCD}

\bibitem[B98]{B98}
{\sc Y. Berkovich}, Subgroups with the character restriction property and related topics. \emph{Houston J. Math.} \textbf{24} (1998), no. 4, 631--638.

\bibitem[CN08]{CN08}
\textsc{P. Centella}, \textsc{G. Navarro}, Correspondences between constituents of projective characters, \emph{Arch. Math. (Basel)} \textbf{90} (2008), no. 4, 289--294.
\bibitem[I84]{I84}
\textsc{I. M. Isaacs}, Characters of $\pi$-separable groups, \emph{J. Algebra} \textbf{86} (1984), no. 1, 98--128.

\bibitem[I86]{I86}
\textsc{I. M. Isaacs}, Subgroups with the character restriction property, \emph{J. Algebra} \textbf{100} (1986), no. 2, 403--420.

\bibitem[I06]{I06}
\textsc{I. M. Isaacs}, \emph{Character Theory of Finite Groups}, AMS-Chelsea, Providence, 2006.

\bibitem[I18]{I18}
\textsc{I. M. Isaacs}, \emph{Characters of Solvable Groups}, Mathematical Surveys and Monographs \textbf{189}, American Mathematical Society, 2018.

\bibitem[IN05]{IN05}
\textsc{I. M. Isaacs} and \textsc{G. Navarro}, A characteristic subgroup and kernels of Brauer characters, \emph{Bull. Aust. Math. Soc.} \textbf{72} (2005), no. 3, 381--384.

\bibitem[N98]{N98}
\textsc{G. Navarro}, \emph{Characters and Blocks of Finite Groups}, Cambridge University Press, 1998.

\bibitem[N18]{N18}
\textsc{G. Navarro}, \emph{Character Theory and the McKay Conjecture}, Cambridge University Press, Cambridge, 2018.

\bibitem[T64]{T64}
\textsc{J. Tate}, Nilpotent quotient groups, \emph{Topology} \textbf{3} (1964), 109--111.

\end{thebibliography}
\end{document}